\documentclass[11pt,reqno]{amsart}
\usepackage{amssymb,latexsym}
\usepackage{graphicx}
\usepackage{url}
\usepackage{graphicx}
\usepackage{mathrsfs}
\usepackage{graphicx}
\usepackage{amsmath}
\usepackage{amssymb,amsfonts,wasysym,latexsym}
\usepackage{amsthm}
\usepackage[autostyle]{csquotes}
\usepackage{mathcomp,stmaryrd}
\usepackage{blkarray}
\usepackage{mathtools}
\usepackage{tikz}
\usetikzlibrary{decorations.pathreplacing}
\usepackage{times}
\usepackage{cleveref}

\calclayout

\newcommand{\e}{{\varepsilon}}

\theoremstyle{plain}
\numberwithin{equation}{section}
\newtheorem{thm}{Theorem}[section]
\newtheorem{theorem}[thm]{Theorem}
\newtheorem{lemma}[thm]{Lemma}
\newtheorem{example}[thm]{Example}
\newtheorem{definition}[thm]{Definition}
\newtheorem{proposition}[thm]{Proposition}
\newtheorem{corollary}[thm]{Corollary}
\newtheorem{remark}[thm]{Remark}

\begin{document}

\title{Zeckendorf representations and mixing properties of sequences}
\author{Neil Ma\~{n}ibo}
\address{Fakult\"at f\"ur Mathematik\\
               Universit\"at Bielefeld\\
               Universit\"atsstra\ss e 25 \\
               33615 Bielefeld, Germany}
\email{cmanibo@math.uni-bielefeld.de}
\author{Eden Delight P. Miro}
\address{Department of Mathematics\\
	Ateneo de Manila University\\
	Loyola Heights\\
	1108 Quezon City, Philippines}
\email{eprovido@ateneo.edu}
\author{Dan Rust}
\address{Fakult\"at f\"ur Mathematik\\
               Universit\"at Bielefeld\\
               Universit\"atsstra\ss e 25 \\
               33615 Bielefeld, Germany}
\email{drust@math.uni-bielefeld.de}
\author{Gwendolyn S. Tadeo}
\address{Department of Mathematics\\
                Saint Louis University\\
                Baguio City\\
                Philippines}
\email{gwendolyn.tadeo@obf.ateneo.edu}

\thanks{Research supported in part by the German Research Foundation (DFG) via the Collaborative Research Centre (CRC 1283), the Alexander von Humboldt Foundation and the Commission on Higher Education of the Philippines (CHED)}

\begin{abstract}
We use generalised Zeckendorf representations of natural numbers to investigate mixing properties of symbolic dynamical systems. The systems we consider consist of bi-infinite sequences associated with so-called random substitutions.
We focus on random substitutions associated with the Fibonacci, tribonacci and metallic mean numbers and take advantage of their respective numeration schemes.
\end{abstract}

\maketitle

\section{Introduction}
The Zeckendorf (or Zeckendorf--Lekkerkerker) representation of a natural number \cite{Brown1964,Lekkerkerker}  is a special case of its Ostrowski numerations \cite{Ostrowski} and allows one to write every natural number uniquely as a sum of distinct non-consecutive Fibonacci numbers. Various generalisations of the Zeckendorf representation exist for other sequences of numbers originating from linear relations \cite{Brown,FibHigher,F:numeration,Hoggatt}. The Zeckendorf representation finds applications in many areas, ranging from combinatorial game theory \cite{W:fibonacci-nim} to error-insensitive data compression \cite{WKP:fib-encoding} and mathematical magic tricks \cite{J:magic}. Although Fibonacci numbers and the golden ratio feature heavily in the study of dynamical systems, specific applications of Zeckendorf's theorem are less prevalent.

In studying mixing properties of some symbolic dynamical systems, the fourth author noticed in her thesis \cite{Tadeo2019} the apparent utility of Zeckendorf representations in representing the lengths of partial orbits in these systems. The initial goal was to use this convenient interplay between properties of Fibonacci numbers and symbolic dynamics to prove that the systems under consideration were topologically mixing. In fact, we eventually proved that this is not the case \cite{MRST}. Even though these dynamical systems were found to be non-mixing, Zeckendorf's Theorem can still be used to prove a weaker version of mixing which we have chosen to call \emph{semi-mixing}. We give a self-contained and streamlined account of that approach here.

The dynamical systems under consideration are systems whose state-space is a particular collection of bi-infinite sequences over a finite alphabet and whose action is given by shifting a sequence to the left. These \emph{shift spaces} or \emph{subshifts} are the primary object of study for symbolic dynamicists and can be thought of as universal objects in the study of discrete dynamical systems in general \cite{LM:book}. The dynamical behaviour of the system can vary greatly depending on how the bi-infinite sequences in our subshift are generated. Our sequences are generated by \emph{random substitutions} \cite{RS}, a generalisation of the more classical notion of a substitution \cite{BG:book,Fogg}. In particular, we have chosen to focus our attention on random substitutions which are closely related to the Fibonacci sequence and its generalisations, the \emph{tribonacci} and \emph{metallic mean} sequences.

In Section \ref{SEC:zeckendorf}, we recall Zeckendorf's theorem and discuss its generalisations to sequences coming from other linear recurrence relations, and so introduce the tribonacci numbers and the metallic mean sequences. In Section \ref{SEC:dynamics}, we give a brief introduction to symbolic dynamics and introduce the notion of random substitutions and their associated subshifts. We also discuss some basic properties of the random Fibonacci, tribonacci and metallic mean substitutions. In Section \ref{SEC:semi-mixing}, we define semi-mixing for general dynamical systems and give an equivalent, more convenient definition for subshifts in terms of word combinatorics. The rest of the section is devoted to proving our main results; namely, that the subshifts associated with the random Fibonacci, tribonacci and metallic mean substitutions are semi-mixing.

\section{Zeckendorf representation and its generalisations}\label{SEC:zeckendorf}

The Fibonacci sequence $\left\{f_n\right\}^{\infty}_{n=0}$ is defined by the recurrence relation $f_n=f_{n-1}+f_{n-2}$ for every integer $n\ge 2$, with $f_0=f_1=1$. 
The \emph{Zeckendorf theorem} (also called the \textit{Zeckendorf--Lekkerkerker theorem}) states that every positive integer can be written uniquely as the sum of distinct non-consecutive Fibonacci numbers \cite{Brown1964, Lekkerkerker}.
\begin{theorem}[{Zeckendorf Theorem}]\label{Zeckendorf}
Every positive integer $n$ has a unique representation in the form
$n=\sum^{r}_{i=1} \e_i f_i$
with $\e_i \in \{0,1\}$ for each $0\leq i \leq r$ and $\e_i \e_{i+1}=0$ for $i\geq 0$. 
\end{theorem}

The unique representation of $n$ as a sum of non-consecutive Fibonacci numbers 
can be obtained using the \textit{greedy algorithm}, which proceeds as follows. 
First, pick the largest Fibonacci number $f_{k_1}$ such that $f_{k_1}\le n$.  
For $n-f_{k_1}>0$, choose again the largest Fibonacci number $f_{k_2}$ such that $f_{k_2}\le n-f_{k_1}$. 
Continuing the process yields $n=f_{k_1}+\cdots+f_{k_r}$ for some $r\in\mathbb{N}$. Note that the Fibonacci numbers in the expansion are distinct since in each step we consider the largest Fibonacci number satisfying the inequality, which is unique. 
Moreover, the process terminates as $n$ is finite and $f_{k_i}$ is always positive. 
Furthermore, by construction, we get a representation $n=f_{k_1}+\cdots+f_{k_r}$ where $k_1>\cdots>k_r \ge 0$.

One generalisation of the Fibonacci sequence is the \textit{tribonacci sequence} $\left\{t_n\right\}^{\infty}_{n=0}$, which is recursively defined via $t_n=t_{n-1}+t_{n-2}+t_{n-3}$ for every integer $n\ge 3$, with $t_0 = 0$, $t_1 = 1$ and $t_2=1$. 

We could instead generalise the Fibonacci sequence in another way, to the so-called \textit{metallic mean sequences}, sometimes called the \emph{noble mean sequences}.
Fix $m \in \mathbb{N}$. The degree-$m$ metallic mean sequence is given by the recursion $z_{n}=mz_{n-1}+z_{n-2}$ for $n \geq 2$, with $z_0=z_1=1$.
There are generalised Zeckendorf theorems associated with both the tribonacci sequence and the metallic mean sequences.
\begin{theorem}[{Tribonacci Zeckendorf theorem \cite{FibHigher}}] Every positive integer $n$ has a unique representation in the form $n = \sum_{i=2}^r \e_i t_i$ with $\e_i \in \{0,1\}$ for each $0 \leq i \leq r$ and $\e_i \e_{i+1} \e_{i+2} = 0$ for $i \geq 0$.
\end{theorem}
\begin{theorem}[{Metallic Mean Zeckendorf Theorem}]\label{Zrnm}
Let $m \geq 1$ be a natural number and let $z_i$ denote the $i$th element of the degree-$m$ metallic mean sequence. 
Every positive integer $n$ has a unique representation of the form
$n=\sum^{r}_{i=1} \e_i z_i$,  with $\e_i \in \{0,\ldots, m\}$ and if $\e_{i+1} = m$, then $\e_{i} = 0$.
\end{theorem}
If the initial condition $z'_0=0$, $z'_1=1$ is instead chosen for the metallic mean sequences, then there is a well-known generalisation of Zeckendorf's theorem, first proved by Hoggatt \cite{Hoggatt}.
The proof for the metallic mean sequences is essentially the same as for Hoggatt's sequences, and so we omit the proof.

We collectively refer to the representations of a natural number obtained above as its \emph{Zeckendorf representation} where it will always be clear from context which of the representations is being used. 
If the natural number $n$ has Zeckendorf representation $\sum_{i=1}^r \e_i f_i$, then we write $[n] = \e_r \e_{r-1} \cdots \e_1$ and call $\e_i$ the $i$-th Zeckendorf digit of $n$ (for the tribonacci Zeckendorf representation, the least significant digit is $\e_2$).

\section{Symbolic Dynamical Systems}\label{SEC:dynamics}
For a comprehensive introduction to symbolic dynamics, we refer the reader to the book of Lind and Marcus \cite{LM:book}.
Let $\mathcal{A}=\{a_1, a_2,\dots,a_\ell\}$ be a finite alphabet. For us, we will normally only consider the alphabets $\{a,b\}$ and $\{a,b,c\}$.
Consider the set $\mathcal{A}^n$ of all words $w = w_1 \cdots w_n$ over $\mathcal{A}$ whose length $|w|$ is equal to $n$ and
let $\mathcal{A}^+=\bigcup_{n=1}^\infty \mathcal{A}^n$ denote the set of all non-empty words. 
A \textit{subword} of a word $w=w_1\dots w_n$ is a word $w_{[i,j]}\coloneqq w_i\cdots w_j$ for some $1\le i\le j\le n$, and we write $w_{[i,j]}\prec w$ to denote the subword relation.
We write $u^i\coloneqq \overbrace{uu\cdots uu}^{i-\text{times}}$ to denote the $i$-fold concatenation of $u$ with itself, with $u^0$ denoting the empty word.
If $\mathcal{A}$ is equipped with the discrete topology, then the set
\[
\mathcal{A}^\mathbb{Z} =\left\{ \cdots x_{-2} x_{-1} \cdot x_{0} x_{1} x_{2} \cdots \mid x_i \in \mathcal{A}\right\}
\]
of all bi-infinite sequences over $\mathcal{A}$ forms a compact metrisable space under the product topology.
The \textit{shift map} $\sigma \colon \mathcal{A}^\mathbb{Z} \rightarrow \mathcal{A}^\mathbb{Z}$ given by $\sigma(x)_n=x_{n+1}$ is a homeomorphism.
We define a \textit{subshift} $X\subseteq \mathcal{A}^\mathbb{Z}$ to be a non-empty closed subspace of $\mathcal{A}^\mathbb{Z}$ that is invariant under the shift action $\sigma$. That is, $\sigma(X)=X$.
The \emph{language} of a subshift $X$ is denoted by $\mathcal{L}$ and consists of all subwords of elements of $X$. That is,
\[
\mathcal{L} = \{u \mid u = x_{[i,j]}, x \in X, i \leq j\}.
\]
If $u \in \mathcal{L}$, then we call $u$ \emph{admitted} by $X$.
We let $\mathcal{L}^n\coloneqq \mathcal{L} \cap \mathcal{A}^n$ denote the set of length-$n$ admitted words.

Let $\mathcal{P}(\mathcal{A}^+)$ denote the power set of $\mathcal{A}^+$.
For sets $A, B \in \mathcal{P}(\mathcal{A}^+)$, we let 
\[
AB = \{uv \mid u \in A, v \in B\}
\]
denote the concatenation of sets of words over $\mathcal{A}$.
A \textit{random substitution} is a map $\vartheta \colon \mathcal{A}\to\mathcal{P}(\mathcal{A}^+)$ such that $\vartheta(a)$ is a non-empty finite set for all $a\in \mathcal{A}$.
Let $w = w_1 \cdots w_k \in \mathcal{A}^+$ be a finite word over $\mathcal{A}$. We extend $\vartheta$ to finite words by defining $\vartheta \colon \mathcal{A}^+ \to \mathcal{P}(\mathcal{A}^+)$ via the concatenation of sets $\vartheta(w) \coloneqq \vartheta(w_1) \cdots \vartheta(w_k)$ and we further extend to $\mathcal{P}(\mathcal{A}^+)$ by $\vartheta(A) = \cup_{w \in A}\vartheta(w)$.
A word $w\in \vartheta(v)$ is called a \textit{realisation of $\vartheta$} on the word $v \in \mathcal{A}^+$. 
We say that a word $u\in \mathcal{A}^+$ is $\vartheta$\textit{-legal} if there is a natural number $k$, a letter $a \in \mathcal{A}$ and a word $w \in \vartheta^k(a)$ such that $u\prec w$.
If the set of words $\vartheta(a)$ is a singleton for each letter $a \in \mathcal{A}$,  $\vartheta$ is called \textit{deterministic} and corresponds to the classical notion of a \emph{substitution} \cite{BG:book}.

Given a random substitution $\vartheta$, we are interested in the closed, shift-invariant subspace of $\mathcal{A}^\mathbb{Z}$ called the \emph{random substitution subshift} (RS-subshift) $X_\vartheta$ associated with $\vartheta$, defined by 
\[
X_\vartheta \coloneqq \left\{x\in \mathcal{A}^\mathbb{Z} \mid  u \prec x \Rightarrow u \mbox{ is }\vartheta\mbox{-legal}\right\}.
\]
The tuple $(X_\vartheta, \sigma)$ forms a topological dynamical system. A comprehensive introduction to random substitutions and their subshifts is given in work of Rust and Spindeler \cite{RS}.
We may now introduce our main objects of interest: the random Fibonacci substitution, the random tribonacci substitution, and the random metallic mean substitutions. 

\begin{example}
The \emph{random Fibonacci substitution} introduced by Godr\`{e}che and Luck \cite{Godreche} is given by $\vartheta_{1}\colon a\mapsto \{ab, ba\}, b\mapsto \{a\}$.
We let $X_1\coloneqq X_{\vartheta_1}$ denote the random Fibonacci subshift.
The set of \emph{level-$n$ inflation words} $\vartheta_1^n(a)$ for random Fibonacci are given by
\[
\begin{array}{l}
\vartheta_1^0(a) = \{a\}, \quad \vartheta_1(a)=\{ab,ba\}, \quad \vartheta^2_1(a)=\{aba,baa,aab\},\\
\vartheta^3_1(a)=\{abaab,ababa,baaab,baaba,aabab,aabba,abbaa,babaa\},\:  \ldots.\\
\end{array}
\]
Note that the length of all level-$n$ inflation words of type $a$ are the same and are equal to the $(n+1)$-th Fibonacci number $f_{n+1}$, thus the name. Further, $|\vartheta_1^n(b)| = f_n$.
\end{example}

In a similar way, we may define a generalisation of the random Fibonacci substitution on three letters associated with the tribonacci sequence.

\begin{example}
The rule $\tau \colon a\mapsto \{ab, ba\}, b\mapsto \{ac,ca\}, c\mapsto \{a\}$ is called the \emph{random tribonacci substitution}. We let $X_\tau$ denote the random tribonacci subshift.
Similarly, the set of level-$n$ inflation words $\tau^n(a)$ for random tribonacci are given by
\begin{equation}\label{eq:recursion tribonacci}
\begin{array}{l}
\tau^0(a) = \{a\}, \quad \tau(a)=\{ab,ba\}, \quad \tau^2(a)=\{abac,abca,baac,baca\},\\
\tau^3(a) = \tau(abac) \cup \tau(abca) \cup \tau(baac) \cup \tau(baca) = \{abacaba, baacaba, \ldots \},\: \cdots.
\end{array}
\end{equation}
It is also easy to see that the length of all level-$n$ inflation words of type $a$ are the same and are equal to the $(n+2)$-th tribonacci number $t_{n+2}$. Further, $|\tau^n(b)| = t_{n} + t_{n+1}$ and $|\tau^n(c)| = t_{n+1}$.
\end{example}
Finally, we may now introduce a family of random substitutions associated with the metallic mean sequences, generalising the random Fibonacci substitution. For a more detailed discussion regarding this family, we refer the reader to the thesis of Moll \cite{Moll}. 
\begin{example}
Fix $m>1$. The degree-$m$ \emph{random metallic mean substitution} is given by
\[
\vartheta_{m}\colon a\mapsto \{a^iba^{m-i} \mid 0\leq i \leq m\},\, b\mapsto \{a\}.
\]
We let $X_m \coloneqq X_{\vartheta_m}$ denote the random metallic mean subshift of degree $m$.
As an example, when $m = 2$, the random metallic mean substitution is given by $\vartheta_2 \colon a \mapsto \{aab, aba, baa\}, b \mapsto \{a\}$ and has level-$n$ inflation words
\[
\begin{array}{l}
\vartheta_2^0(a) = \{a\}, \quad \vartheta_2^1(a)=\{aab, aba, baa\},\\
\vartheta_2^2(a)=\{aabaaba, baaaaba, aabbaaa, baabaaa, aabaaab, baaaaab, aababaa, \ldots\},\: \ldots.\\
\end{array}
\]
It can easily be checked that the uniform length of the level-$n$ inflation words of type $a$ for the degree-$m$ random metallic mean substitution is the $(n+1)$-th degree-$m$ metallic mean number $z_{n+1}$. Further, $|\vartheta^n(b)| = z_{n}$
\end{example}
The substitution matrix of the random Fibonacci substitution $\vartheta_{1}$  is given by 
$\begin{psmallmatrix}
1 & 1 \\
1 & 0 \\
\end{psmallmatrix}$,
with eigenvalues $\lambda_{1,2}=\big\{\frac{1\pm\sqrt{5}}{2}\big\}$, where the Perron--Frobenius (PF) eigenvalue $\lambda_1$ is the golden ratio. 
Since $\lambda_1$ is a Pisot number, i.e., all of its algebraic conjugates lie inside the unit disk, we say that $\vartheta_{1}$ is a \emph{Pisot random substitution}.
The substitution matrices of the random tribonacci substitution and random metallic mean substitutions are 
$\begin{psmallmatrix}
1 & 1 & 1 \\
1 & 0 & 0 \\
0 & 1 & 0 \\
\end{psmallmatrix}$ and 
$\begin{psmallmatrix}
m & 1 \\
1 & 0 \\
\end{psmallmatrix}$, with PF-eigenvalues the tribonacci constant $1.83929\cdots$ and the metallic means $\frac{1}{2}(m+\sqrt{m^2+4})$, respectively.
It is easy to verify that their corresponding PF-eigenvalues are Pisot numbers \cite[Sec.~6]{BG}.  
Hence, they are also Pisot random substitutions.

\section{Semi-mixing of random substitutions}\label{SEC:semi-mixing}
We are now in a position to show that the subshifts associated with the random substitutions introduced in the previous section satisfy a property that is weaker than mixing called semi-mixing, which was first introduced in the thesis of the fourth author \cite{Tadeo2019}.

\begin{definition}
A dynamical system $(X,T)$ is called \emph{semi-mixing} if there exists a proper clopen subset $U\subseteq X$ such that for every open set $V$ in $X$ there exists a natural number $N$ such that for every $n\ge N$, $T^n(V)\cap U\ne \varnothing.$ We say that $(X,T)$ is \textit{semi-mixing with respect to $U$}.
\end{definition}
For a symbolic dynamical system $(X, \sigma)$, we have the following equivalent condition for $(X, \sigma)$ to be semi-mixing.

\begin{proposition}[{\cite[Prop.~4.4]{Tadeo2019}}]\label{pro:semi-mixing}
A subshift $(X,\sigma)$  is \emph{semi-mixing} if and only if there exists a length $\ell$ and a proper subset of length-$\ell$ admitted words $\mathcal{S} \subsetneq \mathcal{L}^\ell$ such that for every word $w\in \mathcal{L}$, there exists a natural number $N$ such that for every $n\ge N$, there exists a word $u$ of length $n$ and a word $s \in \mathcal{S}$ such that $wus \in \mathcal{L}$. 
\end{proposition}
Further, the finite set of words $\mathcal{S}$ can be chosen so that if $(X,\sigma)$ is semi-mixing with respect to a clopen set $U$, then $U = \bigcup_{s \in \mathcal{S}} \mathcal{Z}_u$ where $\mathcal{Z}_u\coloneqq \{x \in X \mid x_{[0,|s|-1]} = s\}$ is the cylinder set at $0$ for the word $s$. As such, we can also unambiguously say that $(X,\sigma)$ is semi-mixing with respect to the finite set of words $\mathcal{S}$.

\begin{remark}
Semi-mixing is a dynamical invariant. That is, it is preserved under  topological conjugacy. Unlike topological mixing, it does not necessarily imply topological transitivity \cite[Ex.~4.1]{Tadeo2019}.
\end{remark}

As the only dynamical action considered here is the shift action $\sigma$, we will suppress the pair notation $(X,\sigma)$ and refer to $X$ unambiguously as the subshift.
In this section, we establish the following main result. 
\begin{theorem}\label{thm:main-semi-mixing}
The following random substitution subshifts are semi-mixing:
\begin{enumerate}
\item The random Fibonacci subshift $X_1$ (with respect to $\mathcal{S}_1 = \left\{ab,ba\right\}$)
\item The random tribonacci subshift $X_{\tau}$ (with respect to $\mathcal{S}_\tau = \left\{ab,ba,ac,ca\right\}$)
\item The random metallic mean subshift $X_m$ (with respect to $\mathcal{S}_m = \left\{a^iba^{m-i} \mid 0\leq i \leq m \right\}$).
\end{enumerate}
\end{theorem}

We should mention that it is actually extremely easy to show that each of these subshifts is semi-mixing with respect to the single word $\mathcal{S} = \{a\}$. This is because, for any legal word $u$, as long as $N$ is large enough, then for $|w| = n \geq N$, either $uw$ is legally followed by an $a$, in which case we are done, or it is legally followed by a $b$ (or a $c$ in the case of tribonacci). Without loss of generality, if $uwb$ is legal, then $b$ is contained in some inflation word of the type $\vartheta(a)$ (no matter which of the three substitutions we choose). Hence, all words in $uw'\vartheta(a)$ are legal for some subword $w' \prec w$. But then $\vartheta(a)$ contains words with an $a$ in all possible positions, and so by choosing some other realisation, we can force an $a$ to appear exactly $n$ places to the right of $u$. So $X$ is semi-mixing with respect to $\mathcal{S} = \{a\}$. 

\subsection{Random Fibonacci}\label{SEC:fib}
We begin with the random Fibonacci subshift $X_{1}$. We will show that $X_1$ is semi-mixing with respect to $\mathcal{S}_{1}=\{ab,ba\}$.
The key observation in our proof will be the fact that, if $[n] = \e_r \cdots \e_1$ is the Zeckendorf representation for $n$ and
\[
u \in \left(\vartheta_1^{r}(b)\right)^{\e_r} \cdots \: \left(\vartheta_1^{\phantom{r}}(b)\right)^{\e_1},
\]
then $[|u|] = [n] = \e_r \cdots \e_1$, and for any $v \in \vartheta_1(u)$, $[|v|] = \e_r \cdots \e_1 0$ and $[|va|] = \e_r \cdots \e_1 1$.
This observation will form the basis for an inductive argument in the length of the Zeckendorf representation of the word $u$. Our result hinges on two technical lemmas regarding legal words and subwords of the above form.
\begin{lemma}\label{LEM:tech1}
If either of $\vartheta_1^p(a)uab$ or $\vartheta_1^p(a)uba$ contain legal words for some word $u$, then both $\vartheta_1^p(a)\vartheta_1(u)ab$ and $\vartheta_1^p(a)\vartheta_1(u)\vartheta_1(b)ba$ contain legal words.
\end{lemma}
\begin{proof}
Suppose that some word in $\vartheta_1^p(a)uab$ is legal.
As $\vartheta_1^p(a)uab$ contains a legal word, then by applying $\vartheta_1$ we see that $\vartheta_1^{p+1}(a)\vartheta_1(u)aba$ contains a legal word.
Note that
\[
\vartheta_1^p(b)\vartheta_1^{p}(a)\vartheta_1(u)aba \subset \vartheta_1^{p+1}(a)\vartheta_1(u)aba
\]
and so $\vartheta_1^{p}(a)\vartheta_1(u)aba$ contains a legal word.
Hence, $\vartheta_1^{p}(a)\vartheta_1(u)ab$ and $\vartheta_1^{p}(a)\vartheta_1(u)\vartheta_1(b)ba$ both contain legal words.

Now suppose that some word in $\vartheta_1^p(a)uba$ is legal.
As $\vartheta_1^p(a)uba$ contains a legal word, then by applying $\vartheta_1$ we see that $\vartheta_1^{p+1}(a)\vartheta_1(u)aba$ contains a legal word.
Note that
\[
\vartheta_1^p(b)\vartheta_1^{p}(a)\vartheta_1(u)aba \subset \vartheta_1^{p+1}(a)\vartheta_1(u)aba
\]
and so $\vartheta_1^{p}(a)\vartheta_1(u)aba$ contains a legal word.
Hence, $\vartheta_1^{p}(a)\vartheta_1(u)ab$ and $\vartheta_1^{p}(a)\vartheta_1(u)\vartheta_1(b)ba$ both contain legal words.
\end{proof}
\begin{lemma}\label{LEM:tech2}
Let $p \geq 1$ be given and let $n\geq 1$ be a natural number whose Zeckendorf representation $[n] = \e_r \cdots \e_1$ has $r=p$ and $\e_r = 1$. Then for all words $u \in (\vartheta_1^p(b))^{\e_p} \cdots (\vartheta_1(b))^{\e_1}$, either $uab$ or $uba$ appear as a prefix of some word in $\vartheta_1^{p+1}(a)$.
\end{lemma}
\begin{proof}
We proceed by induction on $p \geq 1$. As a base case, note that $\vartheta_1^2(a) = \{aba,baa,aab\}$ and for $u=a \in \vartheta_1(b)$ it is true that $uab = aab$ is a prefix of $aab \in \vartheta_1^2(a)$, hence the case $p=1$ is true.
Suppose the statement holds for all $p' \leq p$ and let $n$ have Zeckendorf representation $[n] = \e_{r} \cdots \e_1$ with $r=p+1$ and $\e_{r} = 1$.

\noindent \textbf{Case 1.} If $\e_1 = 0$, then let $m$ have Zeckendorf representation $[m] = \e_r \cdots \e_2$, given by forgetting the final Zeckendorf digit of $n$. By the inductive hypothesis, all words in $(\vartheta_1^p(b))^{\e_r} \cdots (\vartheta_1(b))^{\e_2}$ appear as proper prefixes of words in $\vartheta_1^{p+1}(a)$ with either $ab$ or $ba$ appended to their end. If appended by $ab$, then by applying the substitution $\vartheta_1$ we see that for all words
\[
u \in (\vartheta_1^{p+1}(b))^{\e_r} \cdots (\vartheta_1^2(b))^{\e_2} = (\vartheta_1^{p+1}(b))^{\e_r} \cdots (\vartheta_1^2(b))^{\e_2}(\vartheta_1(b))^{\e_1},
\]
the word $uaba$ and so also $uab$ appears as a prefix of some word in $\vartheta_1^{p+2}(a)$. If appended by $ba$, then similarly, $uaba$ and so $uab$ appears as a prefix of some word in $\vartheta_1^{p+2}(a)$.

\noindent \textbf{Case 2.} If $\e_1 = 1$, then let $m$ have Zeckendorf representation $[m] = \e_r \cdots \e_2$, given by forgetting the final Zeckendorf digit of $n$. By the inductive hypothesis, all words in $(\vartheta_1^p(b))^{\e_r} \cdots (\vartheta_1(b))^{\e_2}$ appear as proper prefixes of words in $\vartheta_1^{p+1}(a)$ wither $ab$ or $ba$ appended to their end.  If appended by $ab$, then by applying the substitution $\vartheta_1$ we see that for all words
\[
u' \in (\vartheta_1^{p+1}(b))^{\e_r} \cdots (\vartheta_1^2(b))^{\e_2},
\]
the word $u'aba$ appears as a prefix of some word in $\vartheta_1^{p+2}(a)$, and note that $u'aba \in u'\vartheta_1(b)ba$, hence for all $u \in (\vartheta_1^{p+1}(b))^{\e_r} \cdots (\vartheta_1(b))^{\e_1}$, the word $uba$ appears as a prefix of some word in $\vartheta_1^{p+2}(a)$. If appended by $ba$, then similarly, $u'aab$ and so $uab$ appears as a prefix of some word in $\vartheta_1^{p+2}(a)$, as required.
\end{proof}
\begin{theorem}\label{THM:semi-fib}
Let $\vartheta_1$ be the random Fibonacci substitution with associated RS-subshift $X_1$.
The subshift $X_1$ is semi-mixing with respect to the set of words $\mathcal{S}_1 = \{ab,ba\}$.
\end{theorem}
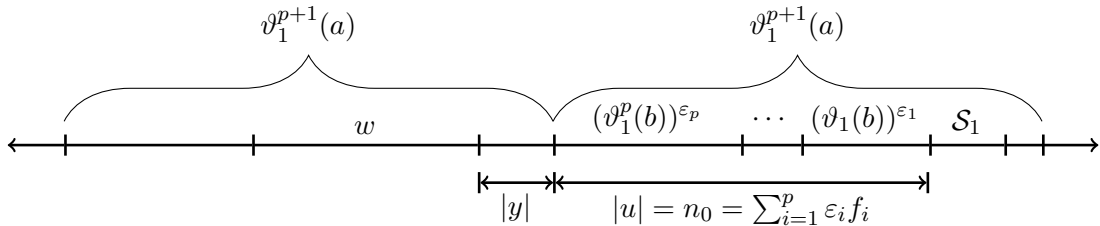
\begin{figure}[h]
            \center
            \begin{tikzpicture}
            \draw[<-,line width= 1.0pt](-1.75,0)--node[above]{}(-1,0);
            \draw[|-,line width= 1.0pt](-1,0)--node[above]{}(1.5,0);
            \draw[|-,line width= 1.0pt](1.5,0)--node[above]{$w$}(4.5,0);
            \draw[|-,line width= 1.0pt](4.5,0)--node[above]{}(5.5,0);
            \draw[|-,line width= 1.0pt](5.5,0)--node[above]{$(\vartheta_1^p(b))^{\e_p}$}(8,0);
            \draw[|-,line width= 1.0pt](8,0)--node[above=3pt]{$\cdots$}(8.8,0);
            \draw[|-,line width= 1.0pt](8.8,0)--node[above]{$(\vartheta_1(b))^{\e_1}$}(10.5, 0);
            \draw[|-,line width= 1.0pt](10.5,0)--node[above]{$\mathcal{S}_1$}(11.5, 0);
            \draw[|-,line width= 1.0pt](11.5,0)--node[above]{}(12,0);
            \draw[|->,line width= 1.0pt](12,0)--node[above]{}(12.75, 0);
            
            \draw[|<->,line width= 1.0pt](4.5,-0.5)--node[below]{$|y|$}(5.5, -0.5);
            \draw[|<->|,line width= 1.0pt](5.5,-0.5)--node[below]{$|u| = n_0 = \sum_{i=1}^p \e_i f_i$}(10.5, -0.5);

            \draw[decoration={brace,amplitude=25pt,raise=9pt},decorate] (-1,0) -- node[above=35pt] {$\vartheta_1^{p+1}(a)$} (5.5,0);
            \draw[decoration={brace,amplitude=25pt,raise=9pt},decorate] (5.5,0) -- node[above=35pt] {$\vartheta_1^{p+1}(a)$} (12,0);
            \end{tikzpicture}
         \caption{\small{Illustration of the construction in the proof of Theorem \ref{THM:semi-fib}.}}
         \label{FIG:main-proof}
\end{figure}
\begin{proof}
Figure \ref{FIG:main-proof} illustrates the construction used in the proof.
Let $w \in \mathcal{L}$ be a legal word. There then exists a power $p \geq 1$ such that $w \prec w' \in \vartheta_1^{p+1}(a)$. Write $w'=xwy$. Let $n_0$ be a natural number whose Zeckendorf representation $[n_0] = \e_r \cdots \e_1$ has $r=p$ and $\e_r = 1$. Pick a word $u \in  (\vartheta_1^{p}(b))^{\e_r} \cdots (\vartheta_1(b))^{\e_1}$ whose length is therefore $n_0$. By Lemma \ref{LEM:tech2}, either $uab$ or $uba$ appears as a prefix of some word in $\vartheta_1^{p+1}(a)$. It follows that either $\vartheta_1^{p+1}(a)uab$ or $\vartheta_1^{p+1}(a)uba$ are subwords of elements of $\vartheta_1^{p+1}(aa)$ and so are legal. So, either $wyuab$ or $wyuba$ is legal. Hence, there exists a word $\hat{u} \coloneqq yu$ of length $N \coloneqq |y|+n_0$ such that either $w \hat{u} ab$ or $w \hat{u} ba$ is legal.

We now proceed by induction.
By Lemma \ref{LEM:tech1}, both $\vartheta_1^{p+1}(a)\vartheta_1(u)ab$ and $\vartheta_1^{p+1}(a)\vartheta_1(u)\vartheta_1(b)ba$ contain legal words and so there exist words $\hat{u}_0$ and $\hat{u}_1$ of lengths $|y| + \e_r \cdots \e_1 0$ and $|y| + \e_r \cdots \e_1 1$ respectively such that both $w\hat{u_0}ab$ and $w\hat{u}_1ba$ are legal. Iterating the procedure, using Lemma \ref{LEM:tech1} together with Zeckendorf's theorem shows that, for all $n \geq N$, there exists a word $\tilde{u}$ of length $n$ such that either $w\tilde{u}ab$ or $w \tilde{u}ba$ is a legal word. It follows that $X_1$ is semi-mixing with respect to $\mathcal{S}_1 = \{ab,ba\}$.
\end{proof}
\begin{remark}
Note that we did not need to use the fact that a Zeckendorf representation is unique. In fact, we could have used the weaker property that the Fibonacci sequence forms a \emph{complete} sequence. That is, for every $n \geq 1$, n can be written as a sum of elements from the sequence, using each element at most once. Brown provided a simple criterion for determining when a sequence is complete \cite{Brown}. This may be useful for extending the above method to other examples where a full Zeckendorf theorem may no longer hold, but the sequence of lengths of level-$n$ inflation words still forms a complete sequence. 
\end{remark}

\subsection{Random tribonacci}\label{SEC:trib}
A similar method to that used for random Fibonacci can also be used to show that the random tribonacci substitution subshift is semi-mixing.
Recall that the random tribonacci substitution is defined by the rule $\tau\colon a\mapsto \{ab, ba\}, b\mapsto \{ac,ca\}, c\mapsto \{a\}$ and its level-$n$ inflation words of type $c$ satisfy $|\tau^n(c)| = t_{n+1}$, where $t_n$ is the $n$-th tribonacci number.

Similar to the case of random Fibonacci, using Zeckendorf representations of natural numbers in terms of tribonacci numbers $n = \sum_{i=2}^r \e_i t_i$, we show that we can find a $\tau$-legal word in $\tau^{r}(a)u$ where $u$ is of the form 
\[
u \in \left(\tau^{r-1}\left( c \right)\right)^{\e_r} \cdots \left(\tau\left( c \right)\right)^{\e_2}
\]
and this word can always be followed by one of the words in $\mathcal{S}_\tau = \{ab,ba,ac,ca\}$.
We then take advantage of the fact that $u$ has length $n$ with $[n] = \e_r \cdots \e_2$ being the Zeckendorf representation of $n$ as a sum of tribonacci numbers.

\begin{lemma}\label{LEM:tech1_trib}
Let $\mathcal{S}_\tau = \{ab,ba,ac,ca\}$.
If, for some word $u$, there is a legal word in 
$\tau^p(a)u\mathcal{S}_\tau$
then at least one of the words in $\tau^p(a)\tau(u)\mathcal{S}_\tau$ is legal and at least one of the words in $\tau^p(a)\tau(u)\tau(c)\mathcal{S}_\tau$ is legal.
\end{lemma}
\begin{proof}
The proof follows in exactly the same fashion as for Lemma \ref{LEM:tech1}, and so we only prove a single instance, when a word in $\tau^p(a)uab$ is legal. The others follow analogously by choosing suitable realisations of $\tau(s)$ for $s \in \mathcal{S}_\tau$.
Suppose that some word in $\tau^p(a)uab$ is legal.
As $\tau^p(a)uab$ contains a legal word, then by applying $\tau$ we see that $\tau^{p+1}(a)\tau(u)abac$ contains a legal word.
Note that
\[
\tau^p(b)\tau^{p}(a)\tau(u)abac \subset \tau^{p+1}(a)\tau(u)abac
\]
and so $\tau^{p}(a)\tau(u)aba$ contains a legal word.
Hence, both
\[
\tau^{p}(a)\tau(u)ab \subset \tau^p(a)\tau(u)\mathcal{S}_\tau\quad \text{ and }\quad \tau^{p}(a)\tau(u)\tau(c)ba\subset \tau^p(a)\tau(u)\tau(c)\mathcal{S}_\tau
\]
contain legal words.

For the other elements $s \in \mathcal{S}_\tau$, it is helpful to choose the following realisations from $\tau(s)$:
\begin{equation}\tag{$\dagger$}
\begin{array}{rl}
ab & \mapsto abac,\\
ba & \mapsto acab,\\
ac & \mapsto aba,\\
ca & \mapsto aba,
\end{array}
\end{equation}
since each of these realisations both have prefixes which are elements of $\mathcal{S}_\tau$ and also prefixes of the form $as' = \tau(c)s'$ for some $s' \in \mathcal{S}_\tau$.
\end{proof}
\begin{lemma}\label{LEM:tech2_trib}
Let $p \geq 2$ be given and let $n\geq 1$ be a natural number whose Zeckendorf representation $[n] = \e_r \cdots \e_2$ has $r=p$ and $\e_r = 1$. Then for all words $u \in (\tau^{p-1}(c))^{\e_p} \cdots (\tau(c))^{\e_2}$, one of the words in $u\mathcal{S}_\tau$ appears as a prefix of some word in $\tau^{p}(a)$.
\end{lemma}
\begin{proof}
The proof follows in exactly the same fashion as for Lemma \ref{LEM:tech2}. Therefore, we invite the reader to fill in the details. Again, it is helpful to choose the same realisations from $\tau(s)$ for each element $s \in\mathcal{S}_\tau$ during the induction step according to ($\dagger$).
\end{proof}
\begin{theorem}\label{THM:semi-trib}
Let $\tau$ be the random tribonacci substitution with associated RS-subshift $X_\tau$.
The subshift $X_\tau$ is semi-mixing with respect to the set of words $\mathcal{S}_\tau = \{ab,ba,ac,ca\}$.
\end{theorem}
\begin{proof}
Let $w \in \mathcal{L}$ be a legal word. There then exists a power $p \geq 2$ such that $w \prec w' \in \tau^{p}(a)$. Write $w'=xwy$. Let $n$ be a natural number whose tribonacci Zeckendorf representation $[n] = \e_r \cdots \e_2$ has $r=p$ and $\e_r = 1$. Pick a word $u \in  (\tau^{p-1}(c))^{\e_r} \cdots (\tau(c))^{\e_2}$ whose length is therefore $n$. By Lemma \ref{LEM:tech2_trib}, one element of $u\mathcal{S}_\tau$ appears as a prefix of some word in $\tau^{p}(a)$. It follows that one word in $\tau^{p}(a)u\mathcal{S}_\tau$ is a subword of some element of $\tau^{p}(aa)$ and so is legal. So, some word in $wyu\mathcal{S}_\tau$ is legal. Hence, there exists a word $\hat{u} \coloneqq yu$ of length $N \coloneqq |y|+n$ such that a word in $w \hat{u} \mathcal{S}_\tau$ is legal.

We now proceed by induction.
By Lemma \ref{LEM:tech1_trib}, there are legal words in both $\tau^{p}(a)\tau(u)\mathcal{S}_\tau$ and $\tau^{p}(a)\tau(u)\tau(c)\mathcal{S}_\tau$ and and so there exist words $\hat{u}_0$ and $\hat{u}_1$ of lengths $|y| + \e_r \cdots \e_2 0$ and $|y| + \e_r \cdots \e_2 1$ respectively such that both $w\hat{u_0}\mathcal{S}_\tau$ and $w\hat{u}_1\mathcal{S}_\tau$ contain legal words. Iterating the procedure, using Lemma \ref{LEM:tech1_trib} together with Zeckendorf's theorem for tribonacci numbers shows that, for all $n \geq N$, there exists a word $\tilde{u}$ of length $n$ such that some word in $w\tilde{u}\mathcal{S}_\tau$ is legal. It follows that $X_\tau$ is semi-mixing with respect to $\mathcal{S}_\tau = \{ab,ba,ac,ca\}$.
\end{proof}
Given the above, it is not hard to see how to generalise Theorem \ref{THM:semi-trib} to all random $k$-bonacci substitutions, random substitutions on the alphabet $\{a_1,\ldots, a_k\}$ given by
\[
\tau_k \colon \begin{cases}
a_i \mapsto \{a_1a_{i+1}, a_{i+1}a_1\},& 1 \leq i \leq k-1,\\
a_k \mapsto \{a_1\}, &
\end{cases}
\]
with $k=2$ corresponding to Fibonacci and $k=3$ corresponding to tribonacci. For each $k$, there is a corresponding Zeckendorf theorem, as described by Carlitz et al.~\cite{FibHigher}. Hence, we have the following corollary.
\begin{corollary}
Let $\tau_k$ be the random $k$-bonacci substitution with associated RS-subshift $X_{\tau_k}$. The subshift $X_{\tau_k}$ is semi-mixing with respect to the set of words $\mathcal{S}_{\tau_k} = \{a_1 a_i, a_i a_1 \mid 1 \leq i \leq k\}$.
\end{corollary}

\subsection{Random metallic means} \label{SEC:metallic}
We now move on to the metallic mean substitutions $\vartheta_m$ for $m \geq 2$ and their subshifts $X_m$. Recall that $\vartheta_m$ is defined by
\[
\vartheta_m \colon a \mapsto \{a^iba^{m-i} \mid 0 \leq i \leq m\}, b \mapsto \{a\}.
\]
As $m$ is fixed, let $\vartheta\coloneqq \vartheta_m$ for the remainder of this section in order to simplify notation.
Recall that the metallic mean sequences $(z_n)_{n \geq 0}$ satisfy the linear recurrence $z_{n} = mz_{n-1} + z_{n-2}$ with initial conditions $z_0 = z_1 = 1$. Recall further that each sequence admits a Zeckendorf theorem (Theorem \ref{Zrnm}) and so every natural number $n$ has a unique representation as a sum $\sum_{i=1}^r \e_i z_i$ with $\e_i \in \{0, \ldots, m\}$ and if $\e_{i+1} = m$ then $\e_i = 0$. Note that our representations can now have Zeckendorf digits greater than 1. For example let $m = 3$ so that
\[
(z_i)_{i \geq 0} = (1, 1, 4, 13, 43, 142, 469, 1420, \ldots)
\]
and let $n = 1404$. Then 
\[
1404 = 3(426) + 0(142) + 2(43) + 3(13) +0(4) +1(1) = 3z_6 + 0z_5 + 2z_4 + 3z_3 + 0z_2 + 1z_1
\]
and so has degree-$3$ Zeckendorf representation $[1404] = 302301$.

As in previous examples, observe that if $[n] = \e_r \cdots \e_1$ is the Zeckendorf representation for $n$ and
\[
u \in \left(\vartheta^{r}(b)\right)^{\e_r} \cdots \: \left(\vartheta^{}(b)\right)^{\e_1},
\]
then $[|u|] = [n] = \e_r \cdots \e_1$, and for any $v \in \vartheta(u)$, $[|va^j|] = \e_r \cdots \e_1 j$ for $0 \leq j \leq m$. Let $\mathcal{S}_m =\{a^iba^{m-i} \mid 0 \leq i \leq m\}$.
\begin{lemma}\label{LEM:tech1_metallic}
If $\vartheta^p(a)u\mathcal{S}_m$ contains a legal word for some word $u$, then for every $0 \leq j \leq m$, each of the sets $\vartheta^p(a)\vartheta(u)(\vartheta(b))^j\mathcal{S}_m$ contains a legal word.
\end{lemma}
The proof for this lemma is similar in spirit to the proof of Lemma \ref{LEM:tech1}, however the generality makes it necessary to consider several cases which differ a fair amount from the construction in the case $m = 1$ for the random Fibonacci. Therefore we spell out every detail of the proof here for completeness.
\begin{proof}
Without loss of generality, we may assume that $m \geq 2$ as the case $m=1$ is covered by Lemma \ref{LEM:tech1}, where $\vartheta = \vartheta_1$ is the random Fibonacci substitution.
Suppose that some word in $\vartheta^p(a)ua^iba^{m-i}$ is legal.

\noindent\textbf{Case 1.} If $i = 0$, then some word in $\vartheta^p(a)uba^m$ is legal, and so by applying $\vartheta$, there must be a legal word in $\vartheta^p(b)(\vartheta^p(a))^{m-1}\vartheta^p(a)\vartheta(u)a(\vartheta(a))^m$. Such a word then contains a subword of the form $\vartheta^p(a) \vartheta(u)a \vartheta(a)\vartheta(a)$. Assuming $j \geq 1$, choose the second-to-rightmost realisation of $\vartheta(a)$ to be $a^{j-1}ba^{m-j+1}$ and the rightmost realisation to be $a^mb$. So some word in 
\[
\vartheta^p(a) \vartheta(u)a a^{j-1}ba^{m-j+1}a^mb = \vartheta^p(a) \vartheta(u)a^jba^{2m-j+1}b
\]
is legal and hence contains a legal subword in 
\[
\vartheta^p(a) \vartheta(u)a^jba^m \subset \vartheta^p(a) \vartheta(u)(\vartheta(b))^j\mathcal{S}_m.
\]
If $j = 0$, then choose the second-to-rightmost realisation of $\vartheta(a)$ to be $ba^m$. So some word in $\vartheta^p(a) \vartheta(u)ab a^m\vartheta(a)$ is legal and hence contains a legal word in
\[
\vartheta^p(a) \vartheta(u)a b a^{m-1} \subset \vartheta^p(a) \vartheta(u)\mathcal{S}_m.
\]

\noindent\textbf{Case 2.} If $i = 1$ then some word in $\vartheta^p(a)uaba^{m-1}$ is legal, and so by applying $\vartheta$, there must be a legal word in $\vartheta^p(b)(\vartheta^p(a))^{m-1}\vartheta^p(a)\vartheta(u)a^jba^{m-j}a(\vartheta(a))^{m-1}$ for any $0 \leq j \leq m$. Such a word then contains a subword of the form $\vartheta^p(a)\vartheta(u)a^jba^{m-j}a\vartheta(a)$. Choose the realisation $a^mb$ in $\vartheta(a)$, and so there is some legal word in $\vartheta^p(a)\vartheta(u)a^jba^{m-j}aa^mb$, hence $\vartheta^p(a)\vartheta(u)a^jba^{m} \in \vartheta^p(a)\vartheta(u)a^j\mathcal{S}_m$ contains a legal word.

\noindent\textbf{Case 3.} If $i \geq 2$ then some word in $\vartheta^p(a)ua^iba^{m-i}$ is legal, and so by applying $\vartheta$, there must be a legal word in $\vartheta^p(b)(\vartheta^p(a))^{m-1}\vartheta^p(a)\vartheta(u)a^jba^{m-j}a^mb(\vartheta(a))^{i-2}\vartheta(b)(\vartheta(a))^{m-i}$ for any $0 \leq j \leq m$. Such a word then contains a subword of the form $\vartheta^p(a)\vartheta(u)a^jba^{m}$ and so there is some legal word in $\vartheta^p(a)\vartheta(u)a^j\mathcal{S}_m = \vartheta^p(a)\vartheta(u)(\vartheta(b))^j\mathcal{S}_m$.
\end{proof}
\begin{lemma}\label{LEM:tech2_metallic}
Let $p \geq 1$ be given and let $n\geq 1$ be a natural number whose Zeckendorf representation $[n] = \e_r \cdots \e_1$ has $r=p$ and $\e_r \geq 1$. Then for all words $u \in (\vartheta^p(b))^{\e_p} \cdots (\vartheta(b))^{\e_1}$, one of the words in $u\mathcal{S}_m$ appears as a prefix of some word in $\vartheta^{p+1}(a)$.
\end{lemma}
\begin{proof}
As in the case for the random Tribonacci, the proof follows in exactly the same fashion as for Lemma \ref{LEM:tech2} with the caveat that there are a few more cases to consider as in the proof of Lemma \ref{LEM:tech1_metallic} in terms of having multiple possible values for $\e_1 \in \{0, \ldots, m\}$ in the inductive step. We invite the reader to fill in the details.
\end{proof}
\begin{theorem}\label{THM:semi-metallic}
Let $\vartheta$ be the random Fibonacci substitution with associated RS-subshift $X_m$.
The subshift $X_m$ is semi-mixing with respect to the set of words $\mathcal{S}_m = \{a^iba^{m-i} \mid 0 \leq i \leq m\}$.
\end{theorem}
\begin{proof}
As with the proofs of Theorems \ref{THM:semi-fib} and \ref{THM:semi-trib}, this follows from a now routine inductive application of the technical Lemmas \ref{LEM:tech1_metallic} and \ref{LEM:tech2_metallic}, together with the symbolic definition of semi-mixing with respect to $\mathcal{S}_m$.
\end{proof}

Doubtless, these methods can sometimes be extended to other classes of substitutions which satisfy the property that $|\vartheta^n(a_k)| = z_n$ for some letter $a_k$ when the sequence $z_n$ admits a Zeckendorf theorem (or is complete). For instance, the $k$-bonacci and metallic mean sequences both generalise to the sequences $z_n$ satisfying
\[
z_n = mz_{n-1} + \sum_{i=2}^k z_{n-i}, \quad z_0= \cdots = z_{k-3}=0, \: z_{k-2}=1, \: z_{k-1}=1
\]
for some $m \geq 1$, and such sequences will also admit a Zeckendorf theorem. It stands to reason that our methods can show that the subshift $X_{k,m}$ associated with the random substitutions $\vartheta_{k,m}$ on the alphabet $\mathcal{A}_k = \{a_1, \ldots, a_k\}$ given by
\[
\vartheta_{k,m} \colon 
\begin{cases}
a_i \mapsto \left\{a_1^i a_{i+1} a_1^{m-i} \mid 0 \leq i \leq m\right\}, & {1 \leq i \leq k-1},\\
a_k \mapsto \{a_1\} &
\end{cases}
\]
will also be semi-mixing with respect to the set of words $\mathcal{S}_{k,m} = \left\{a_1^i a_{j} a_1^{m-i} \mid 0 \leq i \leq m, 2 \leq j \leq k\right\}$. One might call such random substitution \emph{metallic Pisa substitutions} to pay homage to the birth place of Fibonacci.

One characteristic of our methods is that they rely heavily on the fact that our substitutions are random. That is, we make constant use of the fact that letters in our alphabet have multiple images under the random substitution. One is therefore naturally lead to ask whether semi-mixing holds in the deterministic setting for the usual subshift $X_{\operatorname{Fib}}$ associated with the Fibonacci substitution $\theta \colon a \mapsto ab, b \mapsto a$ and similarly for the tribonacci and metallic mean substitutions. We would expect that in the deterministic setting, where there is far less freedom for local exchanges of words, semi-mixing is much more difficult to be satisfied when the subshift is not also topologically mixing.
It is well known that the subshift associated with the Fibonacci substitution, as well as its cousins the irreducible Pisot substitutions, are all not topologically mixing. This follows from the fact that they are minimal system and are not weakly mixing, as mentioned by Kenyon, Sadun and Solomyak \cite{KSS}.
As such, we tentatively conjecture that irreducible Pisot substitutions all have associated subshifts which are not semi-mixing.

\section*{Acknowledgements}

The authors wish to thank Robbert Fokkink for introducing us to generalised Zeckendorf representations and Ethan Akin for helpful discussions regarding the definition of semi-mixing.

\medskip

\noindent MSC2010: 11B39, 37A25, 37B10

\end{document}